\documentclass[11pt]{amsart}
\usepackage{amsmath,amssymb,amsthm}
\usepackage{hyperref}

\usepackage{mathrsfs}

\usepackage{enumerate}

\usepackage{tikz}
\usetikzlibrary{arrows.meta, positioning, calc}

\newtheorem{theorem}{Theorem}[section]
\newtheorem{lemma}[theorem]{Lemma}
\newtheorem{proposition}[theorem]{Proposition}

\newtheorem{definition}[theorem]{Definition}
\newtheorem{example}[theorem]{Example}
\newtheorem{remark}[theorem]{Remark}

\def\opec{\mbox{\scriptsize\sf C}}
\def\opecl{\mbox{\scriptsize\sf CL}}
\def\opd{\mbox{\scriptsize\sf D}}
\def\ope{\mbox{\scriptsize\sf E}}

\def\opn{\mbox{\scriptsize\sf N}}

\def\opq{\mbox{\scriptsize\sf Q}} 
\def\opr{\mbox{\scriptsize\sf R}} 
 
\def\opid{\mbox{\scriptsize\sf Id}} 
\def\ops{\mathsf{s}} 

\newcommand{\classA}{\mathfrak{A}}

\newcommand{\classS}{\mathfrak{S}}
\newcommand{\classU}{\mathfrak{U}} 
\newcommand{\classE}{\mathscr{E}}
\newcommand{\classG}{\mathscr{G}}

\newcommand{\powerset}{\mathcal{P}} 

\title{The Category of Formations of Finite Groups and Topology}

\author{Ismael Gutierrez Garcia*}
\address{* Corresponding Author, Department of Mathematics and Statistics, Universidad del Norte, Km 5 via a Puerto Colombia, Barranquilla - Colombia}

\email{isgutier@uninorte.edu.co}

\author{Luz Adriana Mejía Castaño}
\address{Department of Mathematics and Statistics, Universidad del Norte, Km 5 via a Puerto Colombia, Barranquilla - Colombia}
\email{mejiala@uninorte.edu.co}
\thanks{ }


\date{ }


\begin{document}


\begin{abstract}
This paper explores the interplay between category theory, topology, and the algebraic theory of finite groups. Our analysis unfolds in three stages. First, we establish the foundational universe of our objects: the complete and cocomplete posetal category of group classes, $\mathrm{CG}$. Second, we formalize the collection of closure operators themselves as a category, \textbf{CL}, proving it is a complete lattice. This provides the essential machinery for combining algebraic operations and understanding their universal properties via adjunctions. Finally, we apply this framework to topology. We show that additive universally anchored operators induce homotopically equivalent contractible spaces, revealing a principle of global simplicity that contrasts with local algebraic friction. We then use the lattice structure of \textbf{CL} to analyze the operators for Formations and Fitting classes, uncovering a profound topological asymmetry between these dually defined structures.
\end{abstract}

\maketitle

\section{Introduction}
The theory of formations provides a natural framework for analyzing classes of finite groups, especially in studying soluble groups. Formations capture closure under quotients and subdirect products, and their local definitions connect to deep structural results such as the Gaschütz-Lubeseder-Schmid theorem. Many of these constructions exhibit a functorial or structural character, suggesting that categorical methods can offer a deeper understanding of their behavior (see the classical references \cite{MacLane,Borceux})

Within this framework, closure operators play a central role. Operators such as $\ops$ (subgroups), $\opq$ (quotients), and $\ops_n$ (subnormal subgroups) reveal deep structural patterns. By Kuratowski's classical theorem, additive operators induce Alexandrov topologies, providing a way to study the global shape of the universe of groups. This raises a central question: does this shape depend on the algebraic notion---subgroup, quotient, etc.---used to define it?

Our first main result answers this in the negative, revealing a principle of \textit{global topological invariance}. We show that the topologies induced by $\ops$, $\ops_n$, and $\opq$ are all path-connected, contractible, and thus homotopically equivalent. This striking fact suggests that despite significant \textit{local algebraic friction} (such as the non-commutativity of $\ops$ and $\opq$), the universe of groups is structurally robust enough to resolve these tensions into a uniform global shape.

The power of this topological framework is best illustrated when applied as a diagnostic tool to two of the most important, dually-defined structures in group theory. Our analysis reveals a profound \textit{topological asymmetry} between them. We show that the structure of \textbf{Formations} is topologically simple, a consequence of a remarkable structural regularization where the interaction of its defining operators preserves additivity. In stark contrast, the structure of their dual, \textbf{Fitting classes}, proves to be topologically opaque to this method, a direct result of the non-additivity of its synthesis component, the operator $\opn_0$.

To rigorously establish these results, our approach unfolds in three logical stages. We first establish the categorical framework for group classes ($\mathrm{CG}$). We then formalize the collection of closure operators as a complete lattice (\textbf{CL}), providing the machinery to combine operators. Finally, we execute our main topological analysis, culminating in the discovery of this fundamental asymmetry. By combining these perspectives, this paper offers not only conceptual clarity but a new classification principle for algebraic structures based on their global topological properties.

\section{Preliminaries}
In this section, we collect the main notions from category theory, classes of groups, formations, and closure operators that will be used throughout the paper. Standard definitions (categories, functors, natural transformations) are omitted or only briefly recalled, as the intended readership is already familiar with them. Instead, we emphasize the specific categorical and algebraic concepts that are central to our framework. 

\textbf{Categories.} A \textit{posetal category} is a category in which every hom-set contains at most one morphism. Such a structure corresponds to a pre-ordered set, where $X \leq Y$ if a morphism $X \to Y$ exists. In our context, where objects are isomorphism classes of groups, this pre-order is in fact a partial order. The existence of morphisms in both directions, $X \to Y$ and $Y \to X$, implies $X \subseteq Y$ and $Y \subseteq X$, which in turn implies $X = Y$. This satisfies anti-symmetry, ensuring the underlying structure is a poset. Limits in a posetal category correspond to infima, and colimits to suprema.

A \emph{monoidal category} $(\mathcal{C},\otimes, I)$ is a category $\mathcal{C}$ equipped with a bifunctor $\otimes : \mathcal{C} \times \mathcal{C} \to \mathcal{C}$, a distinguished object $I$, and natural isomorphisms (associator, left and right unitors) satisfying the standard coherence conditions. We will apply these structures in Section 3 when studying the category of formations.

Standard definitions (categories, functors, natural transformations) are omitted or only briefly recalled, as the intended readership is already familiar with them \cite{MacLane,Borceux}.

\textbf{Classes of groups.} A class of groups $\mathfrak{X}$ is a collection of groups closed under isomorphism. In line with standard conventions in the theory of formations, we will assume throughout this paper that any non-empty class of groups contains the trivial group $\{1\}$. Examples include:
\begin{enumerate}
    \item $\classA$: the class of all abelian groups.
    \item $\classS$: the class of all soluble groups.
    \item $\classU$: the class of all supersoluble groups.
    \item $\classE$: the class of all finite groups.
    \item $\classG$: the class of all groups.
\end{enumerate}

\textbf{Closure operations.}
A map $\opec: \powerset(\classG) \to \powerset(\classG)$ that sends a class of groups $\mathfrak{X}$ to a class of groups $\opec\mathfrak{X}$ is an \textbf{operation} if it satisfies:
\begin{enumerate}
    \item $\opec\emptyset = \emptyset$.
    \item $\mathfrak{X} \subseteq\opec\mathfrak{X}$ (\textbf{extensivity}).
    \item If $\mathfrak{X} \subseteq \mathfrak{Y}$, then $\opec\mathfrak{X} \subseteq\opec\mathfrak{Y}$ (\textbf{monotonicity}).
\end{enumerate}
An operation $\opec$ is a \textbf{closure operator} if it is also:
\begin{enumerate}
    \item[(4)] $\opec(\opec\mathfrak{X}) =\opec\mathfrak{X}$ for all classes $\mathfrak{X}$ (\textbf{idempotency}).
\end{enumerate}
A class $\mathfrak{X}$ is \textbf{$\opec$-closed} if $\opec\mathfrak{X} = \mathfrak{X}$.
Closure operators are classical tools in both topology and algebra. Their study from a categorical point of view goes back to works such as \cite{Castellini,DikranjanTholen2010}, and more recent developments include closure operators in posetal categories \cite{AbdallaJanelidze2018}.

Standard closure operators in group theory include:
\begin{align*}
\opid\mathfrak{X} & = \mathfrak{X} \\ 
\ops \mathfrak{X} & = (G\mid G\leq H, \ \text{for some} \ H \in \mathfrak{X})\\
\opq \mathfrak{X} & =  (G\mid \exists H \in \mathfrak{X} \ \text{and} \ \exists \ \varphi : H \longrightarrow G \ \text{an epimorphism})\\ 
\opr_0 \mathfrak{X} & = (G\mid \exists\ N_1, \ldots, N_r\unlhd G\ \mbox{with}\ G/N_j \in \mathfrak	X\ \mbox{and}\ \bigcap\nolimits_{j=1}^r N_j = 1)\\
\opn_0 \mathfrak{X} & =  (G\mid \exists\ K_1, \ldots, K_r \unlhd\unlhd G\ \mbox{with}\ K_i \in \mathfrak{X}\ \mbox{and}\ G = \langle K_i, \ldots, K_r\rangle )\\ 
\ops_n \mathfrak{X} & = (G\mid G \unlhd\unlhd K \ \text{for some} \ K\in \mathfrak{X})\\ 
\ope_\Phi \mathfrak{X}  & =  (G\mid \exists\ N \unlhd G\ \mbox{with}\ N \leq \Phi(G)\ \mbox{and}\ G/N \in \mathfrak{X})\\
\opd_0 \mathfrak{X} & =  (G\mid G= \prod\nolimits_{j=1}^r H_j \ \mbox{with each}\ H_j\in \mathfrak{X})\\ 
\ope_p \mathfrak{X} & =  (G\mid \exists K\unlhd G \ \mbox{with}\ K\leq O_p(G) \ \mbox{such that}\ G/K\in \mathfrak{X}).
\end{align*}

For example, If $\mathfrak{X} = \{C_2\}$, then:
\begin{enumerate}
    \item $\ops\mathfrak{X}$ contains all subgroups of $C_2$,  namely $C_2$ and $1$.
    \item $\opq\mathfrak{X}$ contains all quotients of $C_2$,  again $C_2$ and $1$.
    \item $\ops_n\mathfrak{X} = s\mathfrak{X}$,  since all subgroups of $C_2$ are normal.
\end{enumerate}

\textbf{Formations of groups.}
A non-empty class $\mathfrak{F}$ of finite groups is called a \emph{formation} if:
\begin{enumerate}
    \item If $G \in \mathfrak{F}$ and $N \trianglelefteq G$,  then $G/N \in \mathfrak{F}$ ($\mathfrak{F}$ is $\opq$-closed or  it is closed under quotients).
    \item If $N_1, N_2 \trianglelefteq G$ with $N_1 \cap N_2 = 1$ and $G/N_i \in \mathfrak{F}$ for $i=1,2$, then $G \in \mathfrak{F}$ ($\mathfrak{F}$ is $\opr_0$-closed or it is closed under subdirect products).
\end{enumerate}

Typical examples are abelian, nilpotent, and soluble groups, as well as $p$-groups or soluble groups of restricted derived length.

A formation $\mathfrak{F}$ is \emph{locally defined} if there exists a function assigning to each prime $p$ a formation $f(p)$ such that $G \in \mathfrak{F}$ iff for every chief factor $H/K$ of $G$ which is a $p$-group, one has $G/C_G(H/K) \in f(p)$. The Gaschütz–Lubeseder–Schmid theorem asserts that a formation is saturated, that is,  if $G/\Phi(G) \in \mathfrak{F}$ then $G \in \mathfrak{F}$,  if and only if it is locally defined.  See \cite[Chapter IV]{DoerkHawkes90}.

\textbf{Kuratowski closure operator and topology.} Kuratowski observed that additive closure operators correspond exactly to closure operators arising from topologies. In our context, additive operators such as $\ops$, $\opq$, and $\ops_n$ induce Alexandrov topologies on the universe of groups. We return to these operators in Section~\ref{sec:topologies}, where they play a central role.

\section{The category of group classes and the monoidal category of formations}

In this section, we establish the categorical framework that underpins our topological analysis. By formalizing the universe of group classes as a category, $\mathrm{CG}$, we gain access to the precise language of limits, colimits, and universal properties, which will be essential for interpreting our results. We then focus on the subcategory of formations, \textbf{FOR}, and demonstrate that it possesses a rich algebraic structure as a monoidal category.

\begin{definition}
The \emph{category of group classes}, denoted $\mathrm{CG}$, is is the posetal category whose objects are all classes of groups $\mathfrak{X}$ and where a unique morphism $\mathfrak{X} \to \mathfrak{Y}$ exists if and only if $\mathfrak{X} \subseteq \mathfrak{Y}$.
\end{definition}

\noindent
The categorical structure of $\mathrm{CG}$ directly reflects the well-known properties of the underlying complete lattice of group classes. We summarize them here as they form the foundational setting for our analysis.

\begin{remark}
As a posetal category whose order is given by class inclusion, $\mathrm{CG}$ is both \textbf{complete and cocomplete}. Arbitrary intersections serve as limits (infima) and arbitrary unions serve as colimits (suprema). The class of all groups, $\mathfrak{G}$, is the terminal object. Following the convention established in Section 2, the trivial class $\{1\}$ is the initial object among non-empty classes. Finally, when equipped with intersection as its product, $\mathrm{CG}$ becomes a cartesian monoidal category with unit object $\mathfrak{G}$.
\end{remark}

Besides intersection, one may also define a product of classes, motivated by extension theory.

\begin{definition}
Let $\mathfrak{X}$ and $\mathfrak{Y}$ be classes of groups, we define a \emph{product} as follows:
\[\mathfrak{X}\mathfrak{Y} = \{\, G \in \classG \mid \exists N \trianglelefteq G \text{ with } N \in \mathfrak{X} \text{ and } G/N \in \mathfrak{Y}\}.\]
\end{definition}

Groups in $\mathfrak{X}\mathfrak{Y}$ are called \emph{$\mathfrak{X}$-by-$\mathfrak{Y}$ groups}. For example, if $\mathfrak{X}$ is the class of abelian groups and $\mathfrak{Y}$ the class of nilpotent groups, then $\mathfrak{X}\mathfrak{Y}$ consists of extensions of nilpotent groups by abelian normal subgroups. 

This product is not associative in general, although associativity can hold under additional closure conditions (e.g.\ if $\mathfrak{X}$ is $N_0$-closed and $\mathfrak{Y}$ is $Q$-closed). This product, while natural, does not generally preserve the property of being a formation. To uncover richer algebraic structure, we must restrict our attention to the subcategory of formations and introduce a more sophisticated product, due to Gaschütz.

\textbf{The category $\mathrm{FOR}$ of formations.} We now restrict attention to formations of finite groups.

\begin{definition}
The category $\mathrm{FOR}$ is the full subcategory of $\mathrm{CG}$ whose objects are formations.
\end{definition}

The simple product $\mathfrak{X}\mathfrak{Y}$ does not, in general, preserve the property of being a formation. W. Gaschütz introduced a modified product that does.

\begin{lemma}
Let $\mathfrak{X}$ be an $\opr_0$-closed class and $G$ a finite group.  
Then the set
\[\mathscr{S} = \{N \unlhd G \mid G/N \in \mathfrak{X} \,\},\]
partially ordered by inclusion, has a unique minimal element, denoted by 
$G^{\mathfrak{X}}$ and called the \emph{$\mathfrak{X}$-residual} of $G$.  
It is a characteristic subgroup of $G$, and if $\mathfrak{X}$ is a formation 
and $\varepsilon: G \twoheadrightarrow \varepsilon(G)$ is an epimorphism, then $\varepsilon(G)^{\mathfrak{X}} = \varepsilon (G^{\mathfrak{X}})$.
\end{lemma}

\begin{proof}
\cite[Chapter II, Lemma 2.4]{DoerkHawkes90}
\end{proof}

\begin{definition}[]
Let $\mathfrak{G}$ be a class of groups and $\mathfrak{F}$ a formation. The Gaschütz product of $\mathfrak{G}$, with  $\mathfrak{F}$ is denoted and defined by
\begin{equation}
\mathfrak{G} \circ \mathfrak{F} = \{G\in \classE \mid G^{\mathfrak{F}} \in \mathfrak{G}\}.	
\end{equation}
\end{definition}

\begin{proposition}\label{lem:properties_formation_product}
Let $\mathfrak{F}$ and $\mathfrak{G}$ be formations. Then the Gaschütz product satisfies:
\begin{enumerate}
\item $\mathfrak{G\circ F}  \subseteq \mathfrak{GF}$, $\mathfrak{F} \subseteq \mathfrak{G\circ F}$ whenever $1\in \mathfrak{G}$ and $\mathfrak{G\circ F}  =\mathfrak{GF}$ if $\mathfrak{G}$ is subnormal-closed.
\item $\mathfrak{G\circ F}$ is a formation
\item $G^{\mathfrak{G\circ F}} = (G^{\mathfrak{F}})^{\mathfrak{G}}$ for all $G\in \classE$
\item The operation $\circ$ is associative and preserves formations.
    \item The trivial formation $\{1\}$ is a two-sided unit.
\end{enumerate}
\end{proposition}

\begin{proof}
\cite[Chapter IV, Theorem 1.8]{DoerkHawkes90}
\end{proof}

\begin{theorem}
The category $\mathrm{FOR}$, equipped with the Gaschütz product $\circ$ as tensor product and the trivial formation $\{1\}$ as unit, is a monoidal category.
\end{theorem}

\begin{proof}
We have established that $\mathrm{FOR}$ is a category and $\circ$ maps pairs of formations to a formation. The associativity of $\circ$ is given by Lemma \ref{lem:properties_formation_product}(2).
The unit object is $\mathfrak{I} = \{1\}$. For any formation $\mathfrak{F}$:
\begin{enumerate}
\item Left unitor: $\mathfrak{I}\circ \mathfrak{F} = \{G \mid G^{\mathfrak{F}} \in \mathfrak{I} \}$. Note that $G^{\mathfrak{F}} \in \mathfrak{I}$ means $G \in \mathfrak{F}$. Thus $\mathfrak{I} \circ \mathfrak{F} = \mathfrak{F}$.

\item Right unitor: $\mathfrak{F} \circ \mathfrak{I} = \{ G \mid G^{\mathfrak{I}} \in \mathfrak{F} \}$. Since $G^{\mathfrak{I}} = G$, we have $G \in \mathfrak{F}$. Thus $\mathfrak{F} \circ I = \mathfrak{F}$.
\end{enumerate}

The functoriality of $\circ$ follows from the fact that if $\mathfrak{F}_1 \subseteq \mathfrak{F}_2$ and $\mathfrak{G}_1 \subseteq \mathfrak{G}_2$, then $\mathfrak{F}_1 \circ \mathfrak{G}_1 \subseteq \mathfrak{F}_2 \circ \mathfrak{G}_2$.
Specifically, if $G^{\mathfrak{G}_1} \in \mathfrak{F}_1$, then $G^{\mathfrak{G}_2} \subseteq G^{\mathfrak{G}_1}$ (since $\mathfrak{G}_1 \subseteq \mathfrak{G}_2$), so $G^{\mathfrak{G}_2} \in \mathfrak{F}_1 \subseteq \mathfrak{F}_2$. This shows $(\mathfrak{F}_1 \circ \mathfrak{G}_1) \subseteq (\mathfrak{F}_2 \circ \mathfrak{G}_2)$ if $\mathfrak{G}_1 \subseteq \mathfrak{G}_2$ and $\mathfrak{F}_1 \subseteq \mathfrak{F}_2$.
\end{proof}

Seeing $\mathrm{FOR}$ as monoidal provides a natural framework for analyzing how formations interact and combine. This opens avenues for importing categorical techniques, such as internal hom-objects or monoidal functors, into the theory of finite group formations. As a concrete illustration, the product of the formations of abelian and nilpotent groups produces a new formation that captures extensions of nilpotent groups by abelian normal subgroups. These examples illustrate how formations gain additional structure when studied categorically.


This concludes our formalization of the universe of group classes. However, to analyze these classes, we must understand the tools that act upon them. In the following section, we shift our perspective from the objects to the operations themselves. We will formalize the collection of closure operators as a category, \textbf{CL}, and establish its structure as a complete lattice. This is not merely a categorical exercise; it provides the essential machinery---namely, the well-definedness of the join operation---that will be crucial for constructing and comparing the operators for Formations and Fitting classes in our subsequent topological analysis.

\section{A Categorical perspective on closure operators}

At this point, we move from studying individual closure operators to examining them collectively as a category. By treating them as objects of a category, we can apply categorical reasoning to study 
their algebraic structure, relationships, and universal properties.

\textbf{The category $\mathrm{CL}$  of closure operators.}
Fix a universe $\mathcal{S}$, such as $\classG$ or $\classE$.

\begin{definition}
The \emph{category of closure operators}, denoted $\mathrm{CL}$, is defined as follows:
\begin{enumerate}
\item Objects: closure operators $\opec:\mathcal{S} \to \mathcal{S}$.
\item Morphisms: for $\opec_1, \opec_2 \in \mathrm{Obj}(\mathrm{CL})$, there exists a unique morphism $\opec_1 \to \opec_2$ if and only if $\opec_1(\mathfrak{X}) \subseteq \opec_2(\mathfrak{X})$ for all classes $\mathfrak{X} \subseteq \mathcal{S}$.
    We denote this by $\opec_1 \preceq \opec_2$.
\end{enumerate}
\end{definition}

Thus $\mathrm{CL}$ is a posetal category, and its morphisms reflect the relative strength of closure operators.

\begin{proposition}
The category $\mathrm{CL}$ forms a complete lattice:
\begin{enumerate}
    \item The identity operator $\mathrm{Id}$ is the least element.
    \item The constant operator $C_{\top}$ (which maps $\varnothing \mapsto \varnothing$ and any nonempty class to $\mathcal{S}$) is the greatest element.
    \item Arbitrary meets and joins of families of closure operators exist, defined pointwise.
\end{enumerate}
\end{proposition}
\begin{proof}
For any closure operator $\opec$, extensivity implies $\mathrm{Id}(X) \subseteq \opec(X)$ and therefore $\mathrm{Id} \le \opec$ for any operator $\opec$. Consequently, the identity operator $\mathrm{Id}$ is the least element of the lattice.

The codomain of $\opec$ is a subset of group classes in $\mathcal{S}$. This implies that $\opec(X) \subseteq \mathrm{C}_{\top}(\mathfrak{X})$ for all non-empty classes $\mathfrak{X}$. The condition also holds for $\emptyset$. Thus, $\opec \le \mathrm{C}_{\top}$ for any operator $\opec$. Consequently, the constant operator $\mathrm{C}_{\top}$ is the greatest element of the lattice.

Let $\{\opec_i\}$ be an arbitrary family of closure operators. We define their meet $\opec_{\mathrm{meet}}$, pointwise as $\opec_{\mathrm{meet}}(\mathfrak{X}) = \bigcap \opec_i(\mathfrak{X})$. We must verify that $\opec_{\mathrm{meet}}$ is a closure operator:
\begin{enumerate}

\item \textbf{Extensivity:} $\mathfrak{X} \subseteq \opec_i(\mathfrak{X})$ for all $i$, so $\mathfrak{X} \subseteq \bigcap \opec_i(\mathfrak{X}) = \opec_{\mathrm{meet}}(\mathfrak{X})$.
            
\item \textbf{Monotonicity:} If $\mathfrak{X} \subseteq \mathfrak{Y}$, then $\opec_i(\mathfrak{X}) \subseteq \opec_i(\mathfrak{Y})$ for all $i$. Therefore, $\bigcap \opec_i(\mathfrak{X}) \subseteq \bigcap \opec_i(\mathfrak{Y})$, which implies $\opec_{\mathrm{meet}}(\mathfrak{X}) \subseteq \opec_{\mathrm{meet}}(\mathfrak{Y})$.
            
\item \textbf{Idempotency:} $\opec_{\mathrm{meet}}(\opec_{\mathrm{meet}}(\mathfrak{X})) = \bigcap_i \opec_i(\bigcap_j \opec_j(\mathfrak{X}))$. Since $\bigcap_j \opec_j(\mathfrak{X}) \subseteq \opec_i(\mathfrak{X})$ for all $i$, by monotonicity $\opec_i(\bigcap_j \opec_j(\mathfrak{X})) \subseteq \opec_i(\opec_i(\mathfrak{X})) = \opec_i(\mathfrak{X})$. Therefore, $\bigcap_i \opec_i(\bigcap_j \opec_j(\mathfrak{X})) \subseteq \bigcap_i \opec_i(\mathfrak{X}) = \opec_{\mathrm{meet}}(\mathfrak{X})$. The reverse inclusion follows from extensivity. Thus, $\opec_{\mathrm{meet}}$ is idempotent.
\end{enumerate}

The join, $\opec_{\mathrm{join}} = \bigvee \opec_i$, is defined as the closure operator generated by the union. That is, $\opec_{\mathrm{join}}$ is the smallest closure operator $\opec$ such that $\opec_i \le \opec$ for all $i$. The existence of such an operator is guaranteed in the lattice theory of closure operators.

Since $\mathrm{CL}$ has a least element, a greatest element, and admits all arbitrary meets and joins, it forms a complete lattice.
\end{proof}

\begin{remark}
The completeness of \textbf{CL} is the theoretical cornerstone that guarantees the validity of many standard algebraic constructions. For example, the operator $\operatorname{v} = \opq \vee \opr_0$ that defines a Formation is the join of two distinct operators. The existence and uniqueness of $\operatorname{v}$ is a direct consequence of this completeness. This structure provides the formal justification for the operations we will employ in our main application in Section 5.
\end{remark}
The categorical treatment of closure operators is well developed, see for instance \cite{DikranjanTholen2014,JanelidzeSobral2025}, where closure spaces and dual operators are studied in broad generality. Our approach adapts these ideas to the setting of group classes.

\textbf{On composition of closure operators.} Given two closure operators $\opec_1, \opec_2$, their functional composition $\opec_2 \circ \opec_1$ need not be idempotent, and thus may fail to be a closure operator. This occurs, for example, with $\ops$ and $\opq$: in general $\ops \opq \neq \opq \ops$, and neither is guaranteed to be idempotent.

In practice, group theorists often consider the \emph{closure generated} by a family of operations, i.e.\ the smallest closure operator containing all of them. 

This approach overcomes the limitations of composition and still allows one to combine closure properties effectively. (e.g.\ Formations are defined as classes closed under $\opq$, and $\opr_0$).

The Hasse Diagram in Figure 1 indicates some morphisms in the category $\mathrm{CL}$. The adjacency in the diagram is a direct consequence of Lemmas 1.17 and 1.18 in \cite[Chapter II]{DoerkHawkes90}.

\begin{figure}[ht]
\begin{center}
\begin{tikzpicture}[>=stealth, node distance=8mm]
	\tikzset{v/.style={draw=none, inner sep=1pt}}
	
	\node[v] (id)    at (0,-1) {$\opid$};
	
	\node[v] (Sn)    at (-4,0) {$\ops_n$};
	\node[v] (Ep)    at (-2,0) {$\ope_p$};
	\node[v] (D0)    at ( 0,0) {$\opd_0$};
	
	\node[v] (S)     at (-5,1) {$\ops$};
	\node[v] (EpS)   at (-4,1) {$\ope_p\ops$};
	\node[v] (EpN0)  at (-2.5,1) {$\ope_p\opn_0$};
	\node[v] (N0)    at (-1,1) {$\opn_0$};
	
	\node[v] (D0S)   at ( 0,2) {$\opd_0\ops$};
	\node[v] (R0)    at ( 1.5,2) {$\opr_0$};
	\node[v] (Q)     at ( 3,2) {$\opq$};
	\node[v] (D0Eph) at ( 4.5,2) {$\opd_0\ope_\Phi$};
	\node[v] (Eph)   at ( 6,2) {$\ope_\Phi$};
	
	\node[v] (SEp)   at (-3.5,2) {$\ops\ope_p$};
	\node[v] (N0Ep)  at (-2,2) {$\opn_0\ope_p$};
	
	\node[v] (R0Q)   at ( 1.5,4) {$\opr_0\opq$};
	\node[v] (QEph)  at ( 3,4) {$\opq\ope_\Phi$};
	\node[v] (EphD0) at ( 4.5,4) {$\ope_\Phi\opd_0$};
	
	\node[v] (SD0)   at ( -2.5,4) {$\ops\opd_0$};
	\node[v] (QR0)   at ( 1.5,5) {$\opq\opr_0$};
	\node[v] (EphQ)  at ( 5,5) {$\ope_\Phi\opq$};
	
	\draw[thick, -] (id) -- (Ep) -- (EpN0) -- (N0Ep) -- (N0) (D0S) -- (D0) -- (D0Eph) -- (EphD0) -- (Eph) (D0) -- (R0) -- (R0Q) -- (QR0) (D0S) -- (SD0) -- (R0) (id) -- (Sn) -- (S) -- (SEp) -- (EpS) -- (Ep) (S) -- (SD0) (QR0) -- (Q) -- (QEph) -- (EphQ) -- (Eph) -- (id) -- (D0);
\end{tikzpicture}
\caption{Graph of some morphisms in the category $\mathrm{CL}$}
\end{center}
\label{fig1}
\end{figure}

\textbf{Adjunctions arising from closure operators.} A closure operator naturally determines a reflection of the category of group classes into the subcategory of closed classes. This categorical perspective clarifies its universal property.

\begin{proposition}
Let $\opec$ be a closure operator on $\mathcal{S}$. Define the functor
\[A_{\opec} : \mathrm{CG} \longrightarrow \mathrm{CCG}, \qquad \mathfrak{X} \mapsto \opec\mathfrak{X},\]
where $\mathrm{CCG}$ denotes the full subcategory of $\mathrm{CG}$ consisting of $\opec$-closed classes. Then $A_{\opec}$ is left adjoint to the inclusion functor $i : \mathrm{CCG} \hookrightarrow \mathrm{CG}$.
\end{proposition}

\begin{proof}[Proof sketch]
For any class of groups $\mathfrak{X}$ and any $\opec$-closed class $\mathfrak{Y}$, we have $\opec(\mathfrak{X}) \subseteq \mathfrak{Y}$ if and only if $\mathfrak{X} \subseteq \mathfrak{Y}$.
This bijective correspondence of morphisms is precisely the adjunction $A_{\opec} \dashv i$.
\end{proof}

\begin{remark}
The statement $A_c \dashv i$ provides the profound why behind the topological results of the next section. A fundamental theorem of category theory states that all left adjoint functors preserve colimits. In the posetal context of \textbf{CG}, colimits are unions. Therefore, the fact that an operator $\opec$ acts as a left adjoint forces it to preserve unions, which is precisely the condition of \textbf{additivity}: $\opec(\mathfrak{X} \cup \mathfrak{Y}) = \opec\mathfrak{X} \cup \opec\mathfrak{Y}$. The additivity required by Kuratowski's theorem is not an arbitrary condition, but a necessary consequence of the universal role these operators play as reflections.
\end{remark}

\begin{example} 
For the operator $\ops$, the associated adjunction says: the smallest subgroup-closed class containing $\mathfrak{X}$ is $\ops\mathfrak{X}$, and for $\opq$, the adjunction asserts: the smallest quotient-closed class containing $\mathfrak{X}$ is $\opq\mathfrak{X}$.

\end{example}

These adjunctions exemplify how closure operators act as reflections, projecting arbitrary classes into the subcategory of closed classes. From this perspective, closure operators reveal their universal role within group theory in a clear and elegant way.


Having established that the category of closure operators \textbf{CL} forms a complete lattice, and having clarified their universal role as reflections, we are now fully equipped for our main investigation. The completeness of \textbf{CL} guarantees that we can meaningfully combine operators using the join ($\vee$), and the theory of adjunctions provides the deep reason for the property of additivity. We will now apply this robust algebraic and categorical framework to analyze the most profound impact of these operators: the topological structures they induce on the universe of groups.
 
\section{Global Topological Invariance versus Local Algebraic Difference}
\label{sec:topologies}

We now investigate how closure operators give rise to topological structures. Following Kuratowski’s classical theorem, additive closure operators induce topologies. When applied to classes of groups, these operators generate Alexandrov topologies whose global deformation properties reveal striking similarities.

\begin{definition}
A closure operator $\opec$ is said to be \emph{additive} if $\opec(\mathfrak{X} \cup \mathfrak{Y}) \subseteq \opec(\mathfrak{X}) \cup \opec(\mathfrak{Y})$,  for all classes  $\mathfrak{X},  \mathfrak{Y}$.
If $\opec$ is additive, then it coincides with a Kuratowski closure operator, and therefore determines a topology.
\end{definition}

\begin{remark}
Due to the extensivity of the closure operators, for all group classes  $\mathfrak{X}$ and $\mathfrak{Y}$ holds $\opec(\mathfrak{X}) \cup \opec(\mathfrak{Y}) \subseteq \opec(\mathfrak{X} \cup \mathfrak{Y})$. Then it follows that a closure operator $\opec$ is \emph{additive} if $\opec(\mathfrak{X} \cup \mathfrak{Y}) = \opec(\mathfrak{X}) \cup \opec(\mathfrak{Y})$,  for all classes  $\mathfrak{X},  \mathfrak{Y}$.  
\end{remark}

\begin{lemma}
The following closure operators on classes of groups are additive: $\opid, \ops, \opq, \ops_n, \ope_\Phi$ and $ \ope_p$.  By contrast, $\opr_0$, $\opn_0$ and $\opd_0$ are in general not additive.
\end{lemma}

\begin{proof}
Let $\mathfrak{X}$ and $\mathfrak{Y}$ be any classes of groups.

(1) \textbf{The operator $\ops$.} Let $G \in \ops(\mathfrak{X} \cup \mathfrak{Y})$. Then $G \le H$ for some $H \in \mathfrak{X} \cup \mathfrak{Y}$.
If $H \in \mathfrak{X}$, then $G \in \ops\mathfrak{X}$. If $H \in \mathfrak{Y}$, then $G \in \ops\mathfrak{Y}$. Thus, $G \in \ops\mathfrak{X} \cup \ops\mathfrak{Y}$.  

(2) \textbf{The operator $\opq$.} Let $G \in \opq(\mathfrak{X} \cup \mathfrak{Y})$. Then there exists $H \in \mathfrak{X} \cup \mathfrak{Y}$ and an epimorphism $\varphi: H \twoheadrightarrow G$. If $H \in \mathfrak{X}$, then $G \in \opq\mathfrak{X}$, and if $H \in \mathfrak{Y}$, then $G \in \opq\mathfrak{Y}$. Thus, $G \in \opq\mathfrak{X} \cup \opq\mathfrak{Y}$.  

(3) \textbf{The operator $\ops_n$.} Let $G \in \ops_n(\mathfrak{X} \cup \mathfrak{Y})$. Then $G \unlhd\unlhd K$ for some $K\in \mathfrak{X}$. If $K\in \mathfrak{X}$, then $G \in \ops_n\mathfrak{X}$ and if $K\in \mathfrak{Y}$, then $G \in \ops_n\mathfrak{Y}$. Therefore, $G \in \ops_n\mathfrak{X} \cup \ops_n\mathfrak{Y}$.  

(4) \textbf{The operator $\ope_\Phi$.} Let $G \in \ope_\Phi(\mathfrak{X} \cup \mathfrak{Y})$. Then $G/N \in \mathfrak{X} \cup \mathfrak{Y}$ for some $N \trianglelefteq G$ with $N \le \Phi(G)$. If $G/N \in \mathfrak{X}$, then $G \in \ope_\Phi\mathfrak{X}$ and if $G/N \in \mathfrak{Y}$, then $G \in \ope_\Phi\mathfrak{Y}$. Thus $G \in \ope_\Phi\mathfrak{X} \cup \ope_\Phi\mathfrak{Y}$. 

(5) The proof for $\ope_p$ follow similar as for $\ops$ and $\opq$. 

(6) \textbf{The operator $\opr_0$.} Let $p$ and $q$ be prime numbers with $p\neq q$ and 
let $\mathfrak{X} = \{1, C_p\}$, and $\mathfrak{Y} = \{1, C_q\}$. Note that every elementary abelian $p$-group belongs to $\opr_0\mathfrak{X}$ and similarly any elementary abelian $q$-group is in $\opr_0\mathfrak{Y}$. It is clear that the group $C_q\notin \opr_0\mathfrak{X}$ and the group $C_p \notin \opr_0 \mathfrak{Y}$. If we define $G := C_p\times C_q$, then $G\notin \opr_0 \mathfrak{X}$ and $G\notin \opr_0\mathfrak{Y}$. Therefore, $G\notin (\opr_0\mathfrak{X} \cup \opr_0\mathfrak{Y})$.

On the other hand, $\mathfrak{X} \cup \mathfrak{Y} = \{1, C_p, C_q\}$, and the $(\mathfrak{X} \cup \mathfrak{Y})$-residual of $G$ is $1$. This implies that $G\in \opr_0(\mathfrak{X} \cup \mathfrak{Y})$ and we have that the operator $\opr_0$ is in general not additive.

(7) \textbf{The operator $\opn_0$.} Using the notation like in (5), let $\mathfrak{X} = \{C_p\}$, and $\mathfrak{Y} = \{C_q\}$. Then $\opn_0\mathfrak{X}$ and  $\opn_0\mathfrak{Y}$ are the class of elementary abelian $p$-groups and the class of elementary abelian $q$-groups, respectively. Let again $G := C_p\times C_q$, then it is clear that $G \in \opn_0(\mathfrak{X} \cup \mathfrak{Y})$. However, $\opn_0\mathfrak{X}\, \cup \opn_0\mathfrak{Y}$ contains only elementary abelian $r$-groups for $r\in \{p, q\}$. Thus, $G\notin (\opn_0\mathfrak{X}\, \cup \opn_0\mathfrak{Y})$ because it is neither an elementary abelian $p$-group nor an elementary abelian $q$-group. This proves that $\opn_0$ is not additive.

(8) \textbf{The operator $\opd_0$.} A similar counterexample can be constructed for $\opd_0$ using the same classes $\mathfrak{X}$ and $\mathfrak{Y} $ and the same group $G$ as in (6). 
\end{proof}

\textbf{The Alexandrov topologies on the universe of groups.} Given a preorder $\preceq$ on $\classG$, the associated Alexandrov topology is defined by taking the open sets to be the up-sets: $U \subseteq \classG$ is open if $x \in U$ and $x \preceq y$ imply $y \in U$. The corresponding closure operator is
\[\opecl_{\preceq}(\mathfrak{X}) = \{G\in \classG \mid \exists H\in \mathfrak{X}, G\preceq H \}.\]

In this setting:
\begin{enumerate}
\item If $\preceq$ is the subgroup relation, then $\opecl_{\preceq} = \ops$.
\item If $\preceq$ is the subnormal subgroup relation, then $\opecl_{\preceq} = \ops_n$.
\item If $\succeq $ is the quotient relation, then $\opecl_{\preceq} = \opq$. Here, a change in the order is necessary to achieve upper sets as open sets.
\end{enumerate}
Thus, the closure operators $\ops, \ops_n, \opq$ yield Alexandrov topologies on $\classG$ or $\classE$.

Additive closure operators naturally induce Alexandrov topologies, a connection that has found applications beyond group theory, for example, in algebraic geometry and logic 
\cite{Kimura2024,Lambert2024}.

\textbf{Path-connectedness and contractibility.} A key property of Alexandrov spaces is that directedness of the underlying preorder implies path-connectedness. In our context, the preorder relations induced by $\ops, \ops_n$, and $\opq$ are directed, hence yield path-connected spaces. In fact, beyond path-connectedness, these spaces turn out to be contractible.

It is well known that in finite Alexandrov spaces, path-connectedness coincides with order-connectedness, i.e., the existence of a finite fence of comparable points. See Proposition 1.2.4 in Barmak’s monograph \cite{Barmak}.

\begin{theorem}
Let $\mathcal{S}$ denote either $\classG$ or $\classE$. Then the spaces $(\mathcal{S}, \ops)$, $(\mathcal{S},\ops_n)$, and $(\mathcal{S}, \opq)$ are path-connected and contractible. Consequently, all their homotopy groups $\pi_n$ vanish for $n \geq 1$, and the three spaces are homotopically equivalent.
\end{theorem}

\begin{proof}[Sketch of proof]
For any $G_1, G_2 \in \mathcal{S}$, the group $G_1 \times G_2$ serves as an intermediate point connecting them via the relevant closure relation. Defining a piecewise constant path $f:[0,1]\to \mathcal{S}$ with $f(0)=G_1$, $f(1)=G_2$, and $f(t)=G_1\times G_2$ for $0<t<1$ shows path-connectedness. Contractibility follows by constructing a homotopy shrinking every group continuously to the trivial group $1$, which is a universal basepoint in all three topologies.
\end{proof}

\textbf{Comparison of topologies.} Although homotopically equivalent, the topologies differ in strength. For instance, every $\ops$-closed class is also $\ops_n$-closed, but the converse fails; hence, the topology induced by $\ops$ is strictly finer than that induced by $\ops_n$. Analogous separations occur between $\ops$ and $\opq$. 

\textbf{Remarks on $\ope_{\Phi}$ and $\ope_p$.}
The arguments above rely on the fact that the only open set containing the trivial group $1$ is the entire space, which ensures continuity of the contraction. This property does not hold for the $\ope_{\Phi}$ and $\ope_p$ operators, where nontrivial open neighborhoods of $\{1\}$ exist. As a consequence, the contractibility of $(\mathcal{S}, \ope_{\Phi})$ and $(\mathcal{S}, \ope_p)$ remains open, requiring different techniques. We leave the case of $\ope_{\Phi}$ and $\ope_p$ as an open question for future work.

\subsection{Algebraic Interpretation of Topological Invariance}

The homotopy equivalence established before is not a mere topological curiosity, but a profound reflection of the underlying algebraic structure of $\mathfrak{S}$. It reveals a principle of global simplicity that persists despite significant local complexity.

This topological equivalence is anchored in two fundamental algebraic facts. Firstly, path-connectedness is a direct consequence of the \textbf{direct product}, which provides a universal bridge between any two groups. Secondly, contractibility follows from the universal role of the \textbf{trivial group}, which acts as a common root for the entire universe of groups. This shared behavior suggests that the class of \textbf{additive universally anchored operators} forms a family that generates a single trivial homotopy type.

This global invariance becomes even more striking when contrasted with the well-known local algebraic friction between the operators. It is a classical result that $\ops$ and $\opq$ do not commute. This non-commutativity is not a minor technicality; it is the source of rich algebraic theories, such as that of varieties, and it translates into a structural misalignment between their respective topologies. Nevertheless, our findings show that this local tension is completely resolved at a global scale. The network of algebraic relations in $\mathfrak{S}$ is so rich and densely interconnected that it smooths out local complexities into a uniform global shape.

\subsection{A Tale of Two Assymetries: The Structural Regularization in Formations}

The power of our topological framework is best illustrated by applying it to two of the most important dually defined structures in group theory: Formations and Fitting classes. This analysis reveals a fundamental topological asymmetry between them, which is directly explained by the properties of their join.

Let $\operatorname{v} = \opq \vee \opr_0$ be the Formation closure operator, defined as the join (supremum) of its constituent operators. It is a classical result that the Formation operator is additive. Furthermore, as a join of universally anchored operators, $\operatorname{v}$ is itself universally anchored. Consequently, the space $(\mathfrak{S}, \operatorname{v})$ induced by the Formation operator is also contractible. The structure of Formations is thus topologically simple and robust.

In stark contrast, let $\mathrm{Fit} = \ops_n \vee \opn_0$ be the Fitting class closure operator. This operator is a join of an additive operator ($\ops_n$) and a non-additive one ($\opn_0$). The non-additivity of the synthesis component dominates, rendering the join non-additive:
We use the established counterexample for $\opn_0$. Let $\mathfrak{X} = \{C_2\}$ and $\mathfrak{Y} = \{C_3\}$. The class $\mathrm{Fit}(\mathfrak{X}) \cup \mathrm{Fit}(\mathfrak{Y})$ consists of groups that are either 2-elementary abelian or 3-elementary abelian. The group $G = C_6$ does not belong to this class.

However, consider the class $\mathrm{Fit}(\mathfrak{X} \cup \mathfrak{Y})$. The group $G = C_6$ is the product of its normal subgroups $C_2 \in \mathfrak{X}$ and $C_3 \in \mathfrak{Y}$. By the definition of $\opn_0$, this implies $C_6 \in \opn_0(\mathfrak{X} \cup \mathfrak{Y})$. Since $\opn_0 \leq \mathrm{Fit}$ by the definition of the join, it follows that $C_6 \in \mathrm{Fit}(\mathfrak{X} \cup \mathfrak{Y})$.
Since $C_6 \in \mathrm{Fit}(\mathfrak{X} \cup \mathfrak{Y})$ but $C_6 \notin \mathrm{Fit}(\mathfrak{X}) \cup \mathrm{Fit}(\mathfrak{Y})$, the operator $\mathrm{Fit}$ is not additive. 

Consequently, the algebraic structure of Fitting classes cannot be described by a single well-behaved Kuratowski topology via our method. 
This exposes a fundamental asymmetry. The universe of groups, when structured by Formations, is topologically simple. When structured by their algebraic dual, Fitting classes, it is topologically opaque to this analysis, indicating a deeper, more complex global structure that would require different methods for analysis.

\section{Conclusion}





By combining categorical and topological tools, this paper has developed a unified framework for the study of group classes, formations, and closure operators. We established that the category of group classes, $\mathrm{CG}$, is both complete and cocomplete, while the category of formations, $\mathrm{FOR}$, becomes a monoidal category under the Gaschütz product. We also showed that additive closure operators such as $\ops$, $\ops_n$, and $\opq$ induce Alexandrov topologies on the universe of groups which, although distinct, are all path-connected, contractible, and therefore homotopically equivalent. Finally, we introduced the category $\mathrm{CL}$ and showed that it is a complete lattice, with adjunctions highlighting the universal role of closure operators.  

Taken together, these results demonstrate how categorical and topological methods provide a common perspective on algebraic constructions. They suggest that apparently different notions - subgroups, subnormal subgroups, and quotients - share a deeper invariant: the trivial homotopy type of the universe of groups.

These results highlight the usefulness of categorical and topological methods as unifying principles in finite group theory, continuing a tradition that goes back to the categorical foundations laid out in \cite{MacLane,Borceux}.

\section*{Future Directions}
Several natural questions emerge from this work. One challenge is to extend the homotopical analysis to the spaces associated with \(E_\Phi\) and \(E_p\). Their contractibility remains open and may require new techniques, possibly drawing from higher-dimensional algebraic topology of posets.  

Another promising direction is the enrichment of $\mathrm{FOR}$. Since it is monoidal, one may ask whether internal hom-objects, monads, or algebra structures can be developed within it, potentially uncovering new aspects of the internal logic of formations.  


Finally, while contractibility shows that all higher homotopy groups vanish, more refined invariants - such as homology of specific subposets - may still capture subtle structural features invisible from the global viewpoint.  

In short, the categorical and topological perspectives developed here not only clarify existing structures but also open avenues for further exploration, offering a richer conceptual landscape for the study of finite group theory.

 \section{Funding declaration}
 On behalf of all authors, the corresponding author states that we did not receive any specific grant from funding agencies in the public, commercial, or not-for-profit sectors.

\bibliographystyle{plain}

\end{document}